\documentclass{article}
\usepackage[latin1]{inputenc} 
\usepackage[T1]{fontenc} 
\usepackage[english]{babel} 
\usepackage{indentfirst} 
\usepackage{latexsym}
\usepackage{amssymb}
\usepackage{amsmath}
\usepackage{amsthm}
\usepackage{url}

\newtheorem{thm}{Theorem}[section]
\newtheorem{cor}[thm]{Corollary}
\newtheorem{lem}[thm]{Lemma}
\newtheorem{prop}[thm]{Proposition}
\newtheorem{defn}[thm]{Definition}

\newtheorem*{thm*}{Theorem}

\newcommand{\U}{\mathcal{U}}
\newcommand{\N}{\mathbb{N}}
\newcommand{\Z}{\mathbb{Z}}

\newcommand{\V}{\mathcal{V}}

\newcommand{\bN}{\beta\mathbb{N}}

\begin{document}

\begin{center}
\uppercase{\bf Partition Regularity of Nonlinear Polynomials: a Nonstandard Approach}
\vskip 20pt
{\bf Lorenzo Luperi Baglini\footnote{Supported by grant P25311-N25 of the Austrian Science Fund FWF.}}\\
Department of Mathematics, University of Vienna, 1090 Vienna, Austria.\\
{\tt lorenzo.luperi.baglini@univie.ac.at}\\
\end{center}

\begin{abstract}
\noindent In 2011, Neil Hindman proved that for every natural number $n,m$ the polynomial
\begin{center} $\sum\limits_{i=1}^{n} x_{i}-\prod\limits_{j=1}^{m} y_{j}$ \end{center}
has monochromatic solutions for every finite coloration of $\N$. We want to generalize this result to two classes of nonlinear polynomials. The first class consists of polynomials $P(x_{1},...,x_{n},y_{1},...,y_{m})$ of the following kind:
\begin{center} $P(x_{1},...,x_{n},y_{1},...,y_{m})=\sum\limits_{i=1}^{n}a_{i}x_{i}M_{i}(y_{1},...,y_{m})$, \end{center}
where $n,m$ are natural numbers, $\sum\limits_{i=1}^{n}a_{i}x_{i}$ has monochromatic solutions for every finite coloration of $\N$ and the degree of each variable $y_{1},...,y_{m}$ in $M_{i}(y_{1},...,y_{m})$ is at most one. An example of such a polynomial is
\begin{center} $x_{1}y_{1}+x_{2}y_{1}y_{2}-x_{3}$.\end{center}
The second class of polynomials generalizing Hindman's result is more complicated to describe; its particularity is that the degree of some of the involved variables can be greater than one.\\
The technique that we use relies on an approach to ultrafilters based on Nonstandard Analysis. Perhaps, the most interesting aspect of this technique is that, by carefully chosing the appropriate nonstandard setting, the proof of the main results can be obtained by very simple algebraic considerations. 
\end{abstract}

\section{Introduction}
We say that a polynomial $P(x_{1},...,x_{n})$ is partition regular on $\N=\{1,2,...\}$ if whenever the natural numbers are finitely colored there is a monochromatic solution to the equation $P(x_{1},...,x_{n})=0$. The problem of determining which polynomials are partition regular has been studied since Issai Schur's work \cite{rif20}, and the linear case was settled by Richard Rado in \cite{rif17}:
\begin{thm}[Rado]\label{Rado}Let $P(x_{1},...,x_{n})= \sum_{i=1}^{n}a_{i}x_{i}$ be a linear polynomial with nonzero coefficients. The following conditions are equivalent:
\begin{enumerate}
	\item $P(x_{1},...,x_{n})$ is partition regular on $\N$;
	\item there is a nonempy subset $J$ of $\{1,...,n\}$ such that $\sum\limits_{j\in J}a_{j}=0$.
\end{enumerate}
\end{thm}
In his work Rado also characterized the partition regular finite systems of linear equations. Since then, one of the main topics in this field has been the study of infinite systems of linear equations (for a general background on many notions related to this subject see, e.g., \cite{rif79}). From our point of view, one other interesting question (which has also been approached, e.g., in \cite{rif4}, \cite{rif6}) is: which nonlinear polynomials are partition regular?\\
To precisely formalize the problem, we recall the following definitions:
\begin{defn} A polynomial $P(x_{1},...,x_{n})$ is 
\begin{itemize}
	\item {\bfseries partition regular} $($on $\N)$ if for every natural number $r$, for every partition $\N=\bigcup\limits_{i=1}^{r}A_{i}$, there is an index $j\leq r$ and natural numbers $a_{1},...,a_{n}\in A_{j}$ such that $P(a_{1},...,a_{n})=0$;
	\item {\bfseries injectively partition regular}\footnote{Neil Hindman and Imre Leader proved in \cite{rif11} that every linear partition regular polynomial that has an injective solution is injectively partition regular; see also Section 2.1.} $($on $\N)$ if for every natural number $r$, for every partition $\N=\bigcup\limits_{i=1}^{r}A_{i}$, there is an index $j\leq r$ and mutually distinct natural numbers $a_{1},...,a_{n}\in A_{j}$ such that $P(a_{1},...,a_{n})=0$.
	\end{itemize}
\end{defn}
While the linear case is settled, very little is known in the nonlinear case. One of the few exceptions is the multiplicative analogue of Rado's Theorem, that can be deduced from Theorem \ref{Rado} by considering the map $exp(n)=2^{n}$:
\begin{thm}\label{RadoMoltiplicativo} Let $n,m\geq 1$, $a_{1},...,a_{n},b_{1},...,b_{m}> 0$ be natural numbers, and 
\begin{center} $P(x_{1},...,x_{n},y_{1},...,y_{m})= \prod\limits_{i=1}^{n} x_{i}^{a_{i}}-\prod\limits_{j=1}^{m}y_{j}^{b_{j}}$. \end{center}
The following two conditions are equivalent:
\begin{enumerate}
	\item $P(x_{1},...,x_{n},y_{1},...,y_{m})$ is partition regular;
	\item there are two nonempty subsets $I_{1}\subseteq \{1,...,n\}$ and $I_{2}\subseteq \{1,...,m\}$ such that $\sum\limits_{i\in I_{1}}a_{i}=\sum\limits_{j\in I_{2}}b_{j}$.
\end{enumerate}
\end{thm}
As far as we know, perhaps the most interesting result regarding the partition regularity of nonlinear polynomials is the following: 
\begin{thm}[Hindman]\label{ConsHind} For every natural numbers $n,m\geq 1$, with $n+m\geq 3$, the nonlinear polynomial $$\sum\limits_{i=1}^{n}x_{i}-\prod\limits_{j=1}^{m}y_{j}$$ is injectively partition regular. \end{thm}
Theorem \ref{ConsHind} is a consequence of a far more general result that has been proved in \cite{rif13}.\\
The two main results in our paper are generalizations of Theorem \ref{ConsHind}.\\
In Theorem \ref{lev} we prove that, if $P(x_{1},...,x_{n})=\sum\limits_{i=1}^{n} a_{i}x_{i}$ is a linear injectively partition regular polynomial, $y_{1},...,y_{m}$ are not variables of $P(x_{1},...,x_{n})$, and $F_{1},...,F_{n}$ are subsets of $\{1,...,m\}$, the polynomial
\begin{equation*} R(x_{1},...,x_{n},y_{1},...,y_{m})=\sum_{i=1}^{n} a_{i}(x_{i}\cdot\prod_{j\in F_{i}}y_{j}) \end{equation*}
(having posed $\prod\limits_{j\in F_{i}}y_{j}=1$ if $F_{i}=\emptyset$) is injectively partition regular. E.g., as a consequence of Theorem \ref{lev} we have that the polynomial
\begin{equation*} P(x_{1},x_{2},x_{3},x_{4},y_{1},y_{2},y_{3})= 2x_{1}+3x_{2}y_{1}y_{2}-5x_{3}y_{1}+x_{4}y_{2}y_{3} \end{equation*}
is injectively partition regular. The particularity of polynomials considered in Theorem \ref{lev} is that the degree of each of their variables is one. In Theorem \ref{NLP} we prove that, by slightly modifying the hypothesis of Theorem \ref{lev}, we can ensure the partition regularity for many polynomials having variables with degree greater than one: e.g., as a consequence of Theorem \ref{NLP} we get that the polynomial
 \begin{equation*} P(x,y,z,t_{1},t_{2},t_{3},t_{4},t_{5},t_{6})= t_{1}t_{2}x^{2}+t_{3}t_{4}y^{2}-t_{5}t_{6}z^{2} \end{equation*}
is injectively partition regular.\\
The technique we use to prove our main results is based on an approach to combinatorics by means of nonstandard analysis (see \cite{rif84}, \cite{Tesi}): the idea behind this approach is that, as it is well-known, problems related to partition regularity can be reformulated in terms of ultrafilters. Following an approach that has something in common with the one used by Christian W. Puritz in his articles \cite{rif15}, \cite{rif16}, the one used by Joram Hirschfeld in \cite{rif14} and the one used by Greg Cherlin and Joram Hirschfeld in \cite{rif99}, it can be shown that some properties of ultrafilters can be translated and studied in terms of sets of hyperintegers. This can be obtained by associating, in particular hyperextensions $^{*}\N$ of $\N$, to every ultrafilter $\U$ its monad $\mu(\U)$:
\begin{center} $\mu(\U)=\{\alpha\in$$^{*}\N\mid \alpha\in$$^{*}A$ for every $A\in\U\}$, \end{center}
and then proving that some of the properties of $\U$ can be deduced by properties of $\mu(\U)$ (see \cite{Tesi}, Chapter 2). In particular, we prove that a polynomial $P(x_{1},...,x_{n})$ is injectively partition regular if and only if there is an ultrafilter $\U$, and mutually distinct elements $\alpha_{1},...,\alpha_{n}$ in the monad of $\U$, such that $P(\alpha_{1},...,\alpha_{n})=0$.\\
We will only recall the basic results regarding this nonstandard technique, since it has already been introduced in \cite{rif84} and \cite{Tesi} .\\
The paper is organized as follows: the first part, consisting of section 2, contains an introduction that covers all the needed nonstandard results. In the second part, that consists of sections 3 and 4, we apply the nonstandard technique to prove that there are many nonlinear injectively partition regular polynomials.\\
Finally, in the conclusions, we pose two questions that we think to be quite interesting and challenging.
\section{Basic Results and Definitions}
\subsection{Notions about Polynomials}
In this work, by "polynomial" we mean any $P(x_{1},...,x_{n})\in\Z[\mathbf{X}]$, where $\mathbf{X}$ is a countable set of variables, $\wp_{fin}(\mathbf{X})$ is the set of finite subsets of $\mathbf{X}$ and
\begin{center} $\Z[\mathbf{X}]=\bigcup\limits_{Y\in\wp_{fin}(\mathbf{X})}\Z[Y]$. \end{center}
Given a variable $x\in \mathbf{X}$ and a polynomial $P(x_{1},...,x_{n})$, we denote by $\mathbf{d_{P}(x)}$ the degree of $x$ in $P(x_{1},...,x_{n})$.\\

{\bfseries Convention:} When we write $P(x_{1},...,x_{n})$ we mean that $x_{1},...,x_{n}$ are all and only the variables of $P(x_{1},...,x_{n})$: for every variable $x\in\mathbf{X}$, $d_{P}(x)\geq 1$ if and only if $x\in\{x_{1},...,x_{n}\}$. The only exception is when we have a polynomial $P(x_{1},...,x_{n})$ and we consider one of its monomial: in this case, for the sake of simplicity, we write the monomial as $M(x_{1},...,x_{n})$ even if some of the variables $x_{1},...,x_{n}$ may not divide $M(x_{1},...,x_{n})$.\\

Given the polynomial $P(x_{1},...,x_{n})$, we call {\bfseries set of variables} of $P(x_{1},...,x_{n})$ the set $V(P)=\{x_{1},...,x_{n}\}$, and we call {\bfseries partial degree} of $P(x_{1},...,x_{n})$ the maximum degree of its variables.\\
We recall that a polynomial is linear if all its monomials have degree equal to one and that it is homogeneous if all its monomials have the same degree. Among the nonlinear polynomials, an important class for our purposes is the following:
\begin{defn} A polynomial $P(x_{1},...,x_{n})$ is {\bfseries linear in each variable} $($from now on abbreviated as l.e.v.$)$ if its partial degree is equal to one. \end{defn}
Rado's Theorem \ref{Rado} leads to introduce the following definition:
\begin{defn} A polynomial 
\begin{center} $P(x_{1},...,x_{n})= \sum\limits_{i=1}^{k} a_{i}M_{i}(x_{1},...,x_{n})$, \end{center}
where $M_{1}(x_{1},...,x_{n}),...,M_{k}(x_{1},...,x_{n})$ are its distinct monic monomials, satisfies {\bfseries Rado's Condition} if there is a nonempty subset $J\subseteq \{1,...,k\}$ such that $\sum\limits_{j\in J} a_{j}=0$. \end{defn}
We observe that Rado's Theorem talks about polynomials with constant term equal to zero. In fact Rado, in \cite{rif17}, proved that, when the constant term is not zero, the problem of the partition regularity of $P(x_{1},...,x_{n})$ becomes, in some sense, trivial:
\begin{thm}[Rado] Suppose that 
\begin{center} $P(x_{1},...,x_{n})=(\sum\limits_{i=1}^{n} a_{i}x_{i})+c$ \end{center}
is a polynomial with non-zero constant term $c$. Then $P(x_{1},...,x_{n})$ is partition regular on $\N$ if and only if either 
\begin{enumerate}
	\item there exists a natural number $k$ such that $P(k,k,...,k)=0$;
	\item there exists an integer $z$ such that $P(z,z,...,z)=0$ and there is a nonempty subset $J$ of $\{1,...,n\}$ such that $\sum\limits_{j\in J}a_{j}=0$.
\end{enumerate}
\end{thm}
In order to avoid similar problems, we make the following decision: all the polynomials that we consider in this paper have constant term equal to zero.\\
The last fact that we will often use regards the injective partition regularity of linear polynomials. In \cite{rif11} the authors proved, as a particular consequence of their Theorem 3.1, that a linear partition regular polynomial is injectively partition regular if it has at least one injective solution. Since this last condition is true for every such polynomial (except the polynomial $P(x,y)=x-y$, of course), they concluded that every linear partition regular polynomial on $\N$ is also injectively partition regular. We will often use this fact when studying the injective partition regularity of nonlinear polynomials.
\subsection{The Nonstandard Point of View}
In this section we recall the results that allow us to study the problem of partition regularity of polynomials by mean of nonstandard techniques applied to ultrafilters (see also \cite{rif84} and \cite{Tesi}). We suggest \cite{rif12} as a general reference about ultrafilters, \cite{rif1}, \cite{rif3} or \cite{rif19} as introductions to nonstandard methods and \cite{rif5} as a reference for the model theoretic notions that we use.\\
We assume the knowledge of the nonstandard notions and tools that we use, in particular the knowledge of superstructures, star map and enlarging properties (see, e.g., \cite{rif5}). We just recall the definition of superstructure model of nonstandard methods, since these are the models that we use:
\begin{defn} A {\bfseries superstructure model of nonstandard methods} is a triple $\langle \mathbb{V}(X), \mathbb{V}(Y), *\rangle$ where 
\begin{enumerate}
	\item a copy of $\N$ is included in $X$ and in $Y$;
	\item $\mathbb{V}(X)$ and $\mathbb{V}(Y)$ are superstructures on the infinite sets $X$, $Y$ respectively;
	\item $*$ is a proper star map from $\mathbb{V}(X)$ to $\mathbb{V}(Y)$ that satisfies the transfer property.
\end{enumerate}
\end{defn}
In particular, we use single superstructure models of nonstandard methods, i.e. models where $\mathbb{V}(X)=\mathbb{V}(Y)$, which existence is proved in \cite{rif2}, \cite{rif50} and \cite{rif7}. These models have been chosen because they allow to iterate the star map and this, in our nonstandard technique, is needed to translate the operations between ultrafilters in a nonstandard setting.\\
The study of partition regular polynomials can be seen as a particular case of a more general problem:
\begin{defn}[] Let $\mathcal{F}$ be a family, closed under superset, of nonempty subsets of a set $S$. $\mathcal{F}$ is {\bfseries partition regular} if, whenever $S=A_{1}\cup...\cup A_{n}$, there exists an index $i\leq n$ such that $A_{i}\in \mathcal{F}$.
\end{defn}
Given a polynomial $P(x_{1},...,x_{n})$, we have that $P(x_{1},...,x_{n})$ is (injectively) partition regular if and only if the family of subsets of $\N$ that contain a(n injective) solution to $P(x_{1},...,x_{n})$ is partition regular. We recall that partition regular families of subsets of a set $S$ are related to ultrafilters on $S$:
\begin{thm}\label{ultra} Let $S$ be a set, and $\mathcal{F}$ a family, closed under supersets, of nonempty subsets of $S$. Then $\mathcal{F}$ is partition regular if and only if there exists an ultrafilter $\U$ on $S$ such that $\U\subseteq\mathcal{F}$.
\end{thm}
\begin{proof} This is just a slightly changed formulation of Theorem 3.11 in \cite{rif12}.\end{proof}
Theorem \ref{ultra} leads to introduce two special classes of ultrafilters:
\begin{defn} Let $P(x_{1},...,x_{n})$ be a polynomial, and $\U$ an ultrafilter on $\N$. Then: 
\begin{enumerate}
	\item $\U$ is a $\mathbf{\sigma_{P}}${\bfseries-ultrafilter} if and only if for every set $A\in\U$ there are $a_{1},...,a_{n}\in A$ such that $P(a_{1},..,a_{n})=0$;
	\item $\U$ is a $\mathbf{\iota_{P}}${\bfseries-ultrafilter} if and only if for every set $A\in\U$ there are mutually distinct elements $a_{1},...,a_{n}\in A$ such that $P(a_{1},..,a_{n})=0$.
\end{enumerate}
\end{defn}
As a consequence of Theorem \ref{ultra}, it follows that a polynomial $P(x_{1},...,x_{n})$ is partition regular on $\N$ if and only if there is a $\sigma_{P}$-ultrafilter $\U$ on $\N$, and it is injectively partition regular if and only if there is a $\iota_{P}$-ultrafilter on $\N$.\\
The idea behind the research presented in this paper is that such ultrafilters can be studied, with some important advantages, from the point of view of Nonstandard Analysis. The models of nonstandard analysis that we use are the single superstructure models satisfying the $\mathfrak{c}^{+}$-enlarging property. These models allow to associate hypernatural numbers to ultrafilters on $\N$:
\begin{prop} $(1)$ Let $^{*}\N$ be a hyperextension of $\N$. For every hypernatural number $\alpha$ in $^{*}\N$, the set 
\begin{center} $\mathfrak{U}_{\alpha}=\{A\in\N\mid \alpha\in$$^{*}A\}$ \end{center} is an ultrafilter on $\N$.\\
$(2)$ Let $^{*}\N$ be a hyperextension of $\N$ with the $\mathfrak{c}^{+}$-enlarging property. For every ultrafilter $\U$ on $\N$ there exists an element $\alpha$ in $^{*}\N$ such that $\U=\mathfrak{U}_{\alpha}$.\end{prop}
\begin{proof} These facts are proved, e.g., in \cite{rif97} and in \cite{rif98}.\end{proof}
\begin{defn} Given an ultrafilter $\U$ on $\N$, its {\bfseries set of generators} is
\begin{center} $G_{\U}=\{\alpha\in$$^{*}\N\mid \U=\mathfrak{U}_{\alpha}\}$. \end{center}
\end{defn}
E.g., if $\U=\mathfrak{U}_{n}$ is the principal ultrafilter on $n$, then $G_{\U}=\{n\}$.\\
Here a disclaimer is in order: usually, the set $G_{\U}$ is called "monad of $\U$"; in this paper, from this moment on, the monad on $\U$ will be called "set of generators of $\U$" because, as we will show in Theorem \ref{PBT}, many combinatorial properties of $\U$ can be seen as actually "generated" by properties of the elements in $G_{\U}$.\\
The following is the result that motivates our nonstandard point of view:
\begin{thm}[Polynomial Bridge Theorem]\label{PBT} Let $P(x_{1},...,x_{n})$ be a polynomial, and $\U$ an ultrafilter on $\bN$. The following two conditions are equivalent:
\begin{enumerate}
	\item $\U$ is a $\iota_{P}$-ultrafilter;
	\item there are mutually distinct elements $\alpha_{1},...,\alpha_{n}$ in $G_{\U}$ such that $P(\alpha_{1},...,\alpha_{n})=0$.
\end{enumerate}
\end{thm}
\begin{proof} $(1)\Rightarrow (2)$: Given a set $A$ in $\U$, we consider
\begin{center} $S_{A}=\{(a_{1},...,a_{n})\in A^{n}\mid a_{1},...,a_{n}$ are mutually distinct and $P(a_{1},...,a_{n})=0\}$. \end{center}
We observe that, by hypothesis, $S_{A}$ is nonempty for every set $A$ in $\U$, and that the family \{$S_{A}\}_{A\in \U}$ has the finite intersection property. In fact, if $A_{1},...,A_{m}\in\U$, then
\begin{center} $S_{A_{1}}\cap...\cap S_{A_{m}}=S_{A_{1}\cap...\cap A_{m}}\neq\emptyset$.\end{center}
By $\mathfrak{c}^{+}$-enlarging property, the intersection
\begin{center} $S=\bigcap\limits_{A\in\U}$$^{*}S_{A}$\end{center}
is nonempty. Since, by construction, 
\begin{center} "for every $(a_{1},....,a_{n})\in S_{A}$ $a_{1},...,a_{n}$ are mutually distinct and $P(a_{1},...,a_{n})=0$",\end{center}
by transfer it follows 
\begin{center} "for every $(\alpha_{1},...,\alpha_{n})\in$$^{*}S_{A}$ $\alpha_{1},...,\alpha_{n}$ are mutually distinct and $P(\alpha_{1},...,\alpha_{n})=0$".\end{center}
Let $(\alpha_{1},...,\alpha_{n})$ be an element of $S$. As we observed, $P(\alpha_{1},...,\alpha_{n})=0$, $\alpha_{1},...,\alpha_{n}$ are mutually distinct and, by construction, $\alpha_{1},...,\alpha_{n}\in G_{\U}$ since, for every index $i\leq n$, for every set $A$ in $\U$, $\alpha_{i}\in$$^{*}A$.\\
$(2)\Rightarrow (1)$: Let $\alpha_{1},...,\alpha_{n}$ be mutually distinct elements in $G_{\U}$ such that \\
$P(\alpha_{1},...,\alpha_{n})=0$, and let us suppose that $\U$ is not a $\iota_{P}$-ultrafilter. Let $A$ be an element of $\U$ such that, for every mutually distinct $a_{1},...,a_{n}$ in $A\setminus\{0\}$, $P(a_{1},....,a_{n})\neq 0$.\\
Then by transfer it follows that, for every mutually distinct $\xi_{1},...,\xi_{n}$ in $^{*}A$, 
\begin{center} $P(\xi_{1},...,\xi_{n})\neq 0$;\end{center}
in particular, as $G_{\U}\subseteq$$^{*}A$, for every mutually distinct $\xi_{1},...,\xi_{n}$ in $G_{\U}$ we have $P(\xi_{1},...,\xi_{n})\neq 0$, and this is absurd. Hence $\U$ is a $\iota_{P}$-ultrafilter. \end{proof}
{\bfseries Remark 1:} We obtain similar results if we require that only some of the variables take distinct values: e.g., if we ask for solutions where $x_{1}\neq x_{2}$, we have that for every set $A$ in $\U$ there are $a_{1},...,a_{n}$ with $a_{1}\neq a_{2}$ and $P(a_{1},...,a_{n})=0$ if and only if in $G_{\U}$ there are $\alpha_{1},...,\alpha_{n}$ with $\alpha_{1}\neq\alpha_{2}$ and $P(\alpha_{1},...,\alpha_{n})=0$.\\

{\bfseries Remark 2:} The Polynomial Bridge Theorem is a particular case of a far more general result, that we proved in \cite{Tesi} (Theorem 2.2.9) and we called Bridge Theorem. Roughly speaking, the Bridge Theorem states that, given an ultrafilter $\U$ and a first order open formula $\varphi(x_{1},...,x_{n})$, for every set $A\in\U$ there are elements $a_{1},...,a_{n}\in A$ such that $\varphi(a_{1},...,a_{n})$ holds if and only if there are elements $\alpha_{1},...,\alpha_{n}$ in $G_{\U}$ such that $\varphi(\alpha_{1},...,\alpha_{n})$ holds. E.g., every set $A$ in $\U$ contains an arithmetic progression of length 7 if and only if $G_{\U}$ contains an arithmetic progression of length 7.\\

Since, in the following, we use also operations between ultrafilters, we recall a few definitions about the space $\beta\N$ (for a complete tractation of this space, we suggest \cite{rif12}):
\begin{defn} $\beta\N$ is the space of ultrafilters on $\N$, endowed with the topology generated by the family $\langle \Theta_{A}\mid A\subseteq\N\rangle$, where
\begin{center} $\Theta_{A}=\{\U\in\bN\mid A\in\U\}$. \end{center}
An ultafilter $\U\in\bN$ is called {\bfseries principal} if there exists a natural number $n\in\N$ such that $\U=\{A\subseteq\N\mid n\in A\}$.\\
Given two ultrafilters $\U,\V$, $\U\oplus\V$ is the ultrafilter such that, for every set $A\subseteq\N$, 
\begin{center} $A\in\U\oplus\V\Leftrightarrow \{n\in\N\mid\{m\in\N\mid n+m\in A\}\in\V\}\in\U$. \end{center}
Similarly, $\U\odot\V$ is the ultrafilter such that, for every set $A\subseteq\N$, 
\begin{center} $A\in\U\odot\V\mid \{n\in\N\mid\{m\in\N\mid n\cdot m\in A\}\in\V\}\in\U$. \end{center}
An ultrafilter $\U$ is an {\bfseries additive idempotent} if $\U=\U\oplus\U$; similarly, $\U$ is a {\bfseries multiplicative idempotent} if $\U=\U\odot\U$.\end{defn}
To study ultrafilters from a nonstandard point of view we need to translate the operations $\oplus,\odot$ and the notion of idempotent ultrafilter in terms of generators. These translations involve the iteration of the star map, which is possible in single superstructure models $\langle \mathbb{V}(X),\mathbb{V}(X),*\rangle$ of nonstandard methods:
\begin{defn} For every natural number $n$ we define the function 
\begin{equation*} S_{n}:\mathbb{V}(X)\rightarrow\mathbb{V}(X)\end{equation*}
by setting 
\begin{equation*}S_{1}=*\end{equation*}
and, for $n\geq 1$, 
\begin{equation*}S_{n+1}=*\circ S_{n}.\end{equation*}
\end{defn}
\begin{defn} Let $\langle\mathbb{V}(X),\mathbb{V}(X),*\rangle$ be a single superstructure model of nonstandard methods. We call {\bfseries $\omega$-hyperextension} of $\N$, and we denote by $^{\bullet}\N$, the union of all the hyperextensions $S_{n}(\N)$:
\begin{center} $^{\bullet}\N=\bigcup\limits_{n\in\mathbb{N}} S_{n}(\N)$. \end{center}
\end{defn}
Observe that, as a consequence of the Elementary Chain Theorem, $^{\bullet}\N$ is a nonstandard extension of $\N$.\\
To the elements of $^{\bullet}\N$ is associated a notion of "height":
\begin{defn} Let $\alpha\in$$^{\bullet}\N\setminus\N$. The {\bfseries height} of $\alpha$ $($denoted by $h(\alpha))$ is the least natural number $n$ such that $\alpha\in S_{n}(\N)$. \end{defn}
By convention we set $h(\alpha)=0$ if $\alpha\in\N$. We observe that, for every $\alpha\in$$^{\bullet}\N\setminus\N$ and for every natural number $n\in\N$, $h(S_{n}(\alpha))=h(\alpha)+n$, and that, by definition of height, for every subset $A$ of $\N$ and every element $\alpha\in$$^{\bullet}\N$, $\alpha\in$$^{\bullet}A$ if and only if $\alpha\in S_{h(\alpha)}(A)$.\\
A fact that we will often use is that, for every polynomial $P(x_{1},...,x_{n})$ and every $\iota_{P}$-ultrafilter $\U$, there exists in $G_{\U}$ a solution $\alpha_{1},...,\alpha_{n}$ to the equation $P(x_{1},...,x_{n})=0$ with $h(\alpha_{i})=1$ for all $i\leq n$:
\begin{lem}[Reduction Lemma] Let $P(x_{1},...,x_{n})$ be a polynomial, and $\U$ a $\iota_{P}$-ultrafilter. Then there are mutually distinct elements $\alpha_{1},...,\alpha_{n}\in G_{\U}\cap$$^{*}\N$ such that $P(\alpha_{1},...,\alpha_{n})=0$. \end{lem}
\begin{proof} It is sufficient to apply the Polynomial Bridge Theorem to $^{*}\N\subseteq$$^{\bullet}\N$. \end{proof}
Two observations: first of all, the analogue result holds if $\U$ is just a $\sigma_{P}$-ultrafilter; furthermore, for every natural number $m>1$ there are mutually distinct elements of height $m$ in $G_{\U}$ that form a solution to $P(x_{1},...,x_{n})$: if $\alpha_{1},...,\alpha_{n}$ are given by the Reduction Lemma, we just have to take $S_{m-1}(\alpha_{1}),...,S_{m-1}(\alpha_{n})$.\\
The hyperextension $^{\bullet}\N$ provides an useful framework to translate the operations of sum and product between ultrafilters:
\begin{prop}\label{sum} Let $\alpha,\beta\in$$^{\bullet}\N$, $\U=\mathfrak{U}_{\alpha}$ and $\V=\mathfrak{U}_{\beta}$, and let us suppose that $h(\alpha)=h(\beta)=1$. Then:
\begin{enumerate}
	\item for every natural number $n$, $\mathfrak{U}_{\alpha}=\mathfrak{U}_{S_{n}(\alpha)}$;
	\item $\alpha+$$^{*}\beta\in G_{\U\oplus\V}$;
	\item $\alpha\cdot$$^{*}\beta\in\ G_{\U\odot\V}.$
\end{enumerate}
\end{prop}
\begin{proof} These results have been proved in \cite{rif84} and in \cite{Tesi}, Chapter 2.\end{proof}
{\bfseries Remark:} In Proposition \ref{sum} we supposed, for the sake of simplicity, that $h(\alpha)= h(\beta)=1$. If we drop this hypothesis, the thesis in point (2) becomes 
\begin{equation*} \alpha+S_{h(\alpha)}(\beta)\in G_{\U\oplus\V} \end{equation*}
and, in point (3), the thesis becomes 
\begin{equation*} \alpha\cdot S_{h(\alpha)}(\beta)\in G_{\U\odot\V}. \end{equation*}
Here arises a question: can a similar result be obtained for generical hyperextensions of $\N$ (with this we mean an hyperextension where the iteration of the star map is not allowed)? The answer is: yes and no.\\

{\bfseries Yes:} As Puritz proved in (\cite{rif16}, Theorem 3.4), in each hyperextension that satisfies the $\mathfrak{c}^{+}$-enlarging property we can characterize the set of generators of the tensor product $\U\otimes\V$ in terms of $G_{\U},G_{\V}$ for every ultrafilter $\U$ and $\V$, where $\U\otimes\V$ is the ultrafilter on $\N^{2}$ defined as follows:
\begin{center} $\forall A\subseteq\N^{2}, A\in\U\otimes\V\Leftrightarrow\{n\in\N\mid\{m\in\N\mid (n,m)\in A\}\in\V\}\in\U$. \end{center}
\begin{thm}[Puritz]\label{Puritz} Let $^{*}\N$ be a hyperextension of $\N$ with the $\mathfrak{c}^{+}$-enlarging property. For every ultrafilter $\U,\V$ on $\N$, 
\begin{center} $G_{\U\otimes\V}$=$\{(\alpha,\beta)\in$$^{*}\N^{2}\mid \alpha\in G_{\U}, \beta\in G_{\V}, \alpha<er(\beta)\}$, \end{center} 
where
\begin{center} $er(\beta)=\{$$^{*}f(\beta)\mid f\in \mathtt{Fun}(\N,\N),$$^{*}f(\beta)\in$$^{*}\N\setminus\N\}$. \end{center}
\end{thm}
If we denote by $S:\N^{2}\rightarrow\N$ the operation of sum on $\N$ and by $\hat{S}:\bN^{2}\rightarrow\bN^{2}$ its extension to $\bN$, we have that $\U\oplus\V=\hat{S}(\U\otimes\V)$. So by Puritz's Theorem it follows that
\begin{center} $G_{\U\oplus\V}=\{\alpha+\beta\mid \alpha\in G_{\U}, \beta\in G_{\V}, \alpha<er(\beta)\}$. \end{center}
{\bfseries No:} The characterization given by Theorem \ref{Puritz} is, somehow, "implicit": Proposition \ref{sum} gives a procedure to construct, given $\alpha\in G_{\U}$ and $\beta\in G_{\V}$, an element $\gamma\in G_{\U\oplus\V}$ related to both $\alpha$ and $\beta$, and this fact does not hold for Theorem \ref{Puritz}.\\

An important corollary of Proposition \ref{sum} is that we can easily characterize the idempotent ultrafilters in the nonstandard setting:
\begin{prop}\label{idultragen} Let $\U\in\bN$. Then:
\begin{enumerate}
	\item $\U\oplus\U=\U\Leftrightarrow\forall\alpha,\beta\in G_{\U}\cap$$^{*}\N$ $\alpha+$$^{*}\beta\in G_{\U}\Leftrightarrow\exists\alpha,\beta\in G_{\U}\cap$$^{*}\N$ $\alpha+$$^{*}\beta\in G_{\U}$;
	\item $\U\odot\U=\U\Leftrightarrow\forall\alpha,\beta\in G_{\U}\cap$$^{*}\N$ $\alpha+$$^{*}\beta\in G_{\U}\Leftrightarrow\exists\alpha,\beta\in G_{\U}\cap$$^{*}\N$ $\alpha\cdot$$^{*}\beta\in G_{\U}$.
\end{enumerate}
\end{prop}
\begin{proof} The thesis follows easily by points (2) and (3) of Proposition \ref{sum}.\end{proof}
In \cite{rif84} these characterizations of idempotent ultrafilters are used to prove some results in combinatorics, in particular a "constructive" proof of (a particular case of) Rado's Theorem.\\
In the next two sections we show how the nonstandard approach to ultrafilters can be used to prove the partition regularity of particular nonlinear polynomials.
\section{Partition Regularity for a Class of l.e.v. Polynomials}
In \cite{rif6}, P. Csikv\'ari, K. Gyarmati and A. S\'arközy posed the following question (that we reformulate with the terminology introduced in section 2): is the polynomial 
\begin{center} $P(x_{1},x_{2},x_{3},x_{4})=x_{1}+x_{2}-x_{3}x_{4}$ \end{center}
injectively partition regular? This problem was solved by Neil Hindman in \cite{rif13} as a particular case of Theorem \ref{ConsHind}, that we recall:
\begin{thm*} For every natural number $n,m\geq 1$, with $n+m\geq 3$, the nonlinear polynomial $$\sum\limits_{i=1}^{n}x_{i}-\prod\limits_{j=1}^{m}y_{j}$$ is injectively partition regular. \end{thm*}
We start this section by proving the previous theorem using the nonstandard approach to ultrafilters introduced in section 2.\\
A key result in our approach to the partition regularity of polynomials is the following:
\begin{thm}\label{idultra} If $P(x_{1},...,x_{n})$ is an homogeneous injectively partition regular polynomial then there is a nonprincipal multiplicative idempotent $\iota_{P}$-ultrafilter. \end{thm}
\begin{proof} Let 
\begin{center} $I_{P}=\{\U\in\bN\mid \U$ is a $\iota_{P}$-ultrafilter\}.\end{center}
We observe that $I_{P}$ is nonempty since $P(x_{1},...,x_{n})$ is partition regular. By the definition of $\iota_{P}$-ultrafilter, and by Theorem \ref{PBT}, it clearly follows that every ultrafilter in $\U$ is nonprincipal, since $|G_{\U}|=1$ for every principal ultrafilter.\\

{\bfseries Claim:} $I_{P}$ is a closed bilateral ideal in $(\bN,\odot)$.\\

If we prove the claim, the thesis follows by Ellis' Theorem (see \cite{rif9}).\\
$I_{P}$ is closed since, as it is known, for every property $P$ the set 
\begin{center} $\{\U\in\bN\mid \forall A\in\U \ \ A$ satisfies $P\}$ \end{center}
is closed.\\
$I_{P}$ is a bilateral ideal in $(\bN,\odot)$: let $\U$ be an ultrafilter in $I_{P}$, let $\alpha_{1},...,\alpha_{n}$ be mutually distinct elements in $G_{\U}\cap$$^{*}\N$ with $P(\alpha_{1},...,\alpha_{n})=0$ and let $\V$ be an ultrafilter in $\bN$. Let $\beta\in$$^{*}\N$ be a generator of $\V$.\\
By Proposition \ref{sum} it follows that $\alpha_{1}\cdot$$^{*}\beta,...,\alpha_{n}\cdot$$^{*}\beta$ are generators of $\U\odot\V$. They are mutually distinct and, since $P(x_{1},...,x_{n})$ is homogeneous, if $d$ is the degree of $P(x_{1},...,x_{n})$ then
\begin{center} $P(\alpha_{1}\cdot$$^{*}\beta,...,\alpha_{n}\cdot$$^{*}\beta)=$$^{*}\beta^{d} P(\alpha_{1},...,\alpha_{n})=0$.\end{center}
So $\U\odot\V$ is a $\iota_{P}$-ultrafilter, and hence it is in $I_{P}$.\\
The proof for $\V\odot\U$ is completely similar: in this case, we consider the generators $\beta\cdot$$^{*}\alpha_{1},...,\beta\cdot$$^{*}\alpha_{n}$, and we observe that
\begin{center} $P(\beta\cdot$$^{*}\alpha_{1},...,\beta\cdot$$^{*}\alpha_{n})=\beta^{d} P($$^{*}\alpha_{1},...,$$^{*}\alpha_{n})=0$ \end{center}
since, by transfer, if $P(\alpha_{1},...,\alpha_{n})=0$ then $P($$^{*}\alpha_{1},...,$$^{*}\alpha_{n})=0$.\\
So $I_{P}$ is a bilateral ideal, and this concludes the proof.\end{proof}
{\bfseries Remark:} Theorem \ref{idultra} is a particular case of Theorem 3.3.5 in \cite{Tesi}) which, roughly speaking, states that whenever we consider a first order open formula $\varphi(x_{1},...,x_{n})$ that is "multiplicatively invariant" (with this we mean that, whenever $\varphi(a_{1},...,a_{n})$ holds, for every natural number $m$ also $\varphi(m\cdot a_{1},..., m\cdot a_{n})$ holds) the set 
\begin{center} $I_{\varphi}=\{\U\in\bN\mid\forall A\in\U \ \ \exists a_{1},...,a_{n}$ such that $\varphi(a_{1},...,a_{n})$ holds$\}$ \end{center}
is a bilateral ideal in $(\bN,\odot)$. This, by Ellis's Theorem, entails that $I_{\varphi}$ contains a multiplicative idempotent ultrafilter (and we can prove that this ultrafilter can be taken to be nonprincipal). Similar results hold if $\varphi(x_{1},...,x_{n})$ is "additively invariant", and for other similar notions of invariance.\\

As a consequence of Theorem \ref{idultra}, we can reprove Theorem \ref{ConsHind}:
\begin{thm*} For every natural number $n,m\geq 1$, with $n+m\geq 3$, the nonlinear polynomial $$\sum\limits_{i=1}^{n}x_{i}-\prod\limits_{j=1}^{m}y_{j}$$ is injectively partition regular. \end{thm*}
\begin{proof} If $n\geq 2$, $m=1$, the polynomial is $\sum_{i=1}^{n} x_{i}-y$, and we can apply Rado's Theorem. If $n=1, m\geq 2$ the polynomial is $x-\prod\limits_{i=1}^{m}y_{i}$, and we can apply the multiplicative analogue of Rado's Theorem (Theorem \ref{RadoMoltiplicativo}).\\
So, we suppose $n\geq 2,$ $m\geq 2$ and we consider the polynomial 
\begin{center} $R(x_{1},...,x_{n},y)= \sum\limits_{i=1}^{n} x_{i}-y$. \end{center}
By Rado's Theorem, $R(x_{1},...,x_{n},y)$ is partition regular so, as we observed in section 2, since it is linear it is, in particular, injectively partition 
regular. It is also homogeneous, so there exists a multiplicative idempotent $\iota_{R}$-ultrafilter $\U$. Let $\alpha_{1},...,\alpha_{n},\beta$ be mutually distinct elements in $G_{\U}\cap$$^{*}\N$ with $\sum\limits_{i=1}^{n}\alpha_{i}-\beta=0$.\\
Now let
\begin{center} $\eta=\prod\limits_{j=1}^{m} S_{j}(\beta)$. \end{center}
For $i=1,...,n$ we set 
\begin{center} $\lambda_{i}=\alpha_{i}\cdot\eta$ \end{center}
and, for $j=1,...,m$, we set
\begin{center} $\mu_{j}=S_{j}(\beta)$. \end{center}
Now, for $i\leq n, j\leq m$ we set $x_{i}=\lambda_{i}$ and $y_{j}=\mu_{j}$. Since $\U$ is a multiplicative idempotent, all these elements are in $G_{\U}$. Also,
\begin{center} $\sum\limits_{i=1}^{n}\lambda_{i}-\prod\limits_{j=1}^{m}\mu_{j}=\eta(\sum\limits_{i=1}^{n}\alpha_{i}-\beta)=0$, \end{center}
and this shows that $\U$ is a $\iota_{P}$-ultrafilter. In particular $P(x_{1},...,x_{n},y_{1},...,y_{m})$ is injectively partition regular.\end{proof}
These ideas can be slightly modified to prove a more general result:
\begin{defn} Let $m$ be a positive natural number, and let $\{y_{1},...,y_{m}\}$ be a set of mutually distinct variables. For every finite set $F\subseteq\{1,..,m\}$ we denote by $Q_{F}(y_{1},...,y_{m})$ the monomial
\begin{center} $Q_{F}(y_{1},...,y_{m})=\begin{cases} \prod\limits_{j\in F} y_{j}, & \mbox{if  } F\neq \emptyset;\\ 1, & \mbox{if  } F=\emptyset.\end{cases}$ \end{center}
\end{defn}
\begin{thm}\label{lev} Let $n\geq 2$ be a natural number, let $R(x_{1},...,x_{n})=\sum\limits_{i=1}^{n} a_{i}x_{i}$ be a partition regular polynomial, and let $m$ be a positive natural number. Then, for every $F_{1},...,F_{n}\subseteq\{1,..,m\}$ $($with the request that, when $n=2$, $F_{1}\cup F_{2}\neq\emptyset)$, the polynomial
\begin{center} $P(x_{1},...,x_{n},y_{1},...,y_{m})=\sum\limits_{i=1}^{n} a_{i}x_{i}Q_{F_{i}}(y_{1},...,y_{m})$ \end{center}
is injectively partition regular. \end{thm}
\begin{proof} If $n=2$, since in this case we supposed that at least one of the monomials has degree greater than one, we are in a particular case of the multiplicative analogue of Rado's Theorem with at least three variables, and this ensures that the polynomial is injectively partition regular. Hence we can suppose, from now on, $n\geq 3$.\\
Since $R(x_{1},...,x_{n})$ is linear (so, in particular, it is homogeneous) and partition regular, by Theorem \ref{idultra} it follows that there is a nonprincipal multiplicative idempotent $\iota_{R}$-ultrafilter $\U$. Let $\alpha_{1},...,\alpha_{n}\in$$^{*}\N$ be mutually distinct generators of $\U$ such that $R(\alpha_{1},...,\alpha_{n})=0$, and let $\beta\in$$^{*}\N$ be any generator of $\U$. For every index $j\leq m$, we set
\begin{center} $\beta_{j}=S_{j}(\beta)\in G_{\U}$. \end{center}
We observe that, for every index $j\leq m$, $\beta_{j}\in G_{\U}$. We set, for every index $i\leq n$,
\begin{center} $\eta_{i}=\alpha_{i}\cdot (\prod\limits_{j\notin F_{i}} \beta_{j})$. \end{center}
Since $\U$ is a multiplicative idempotent, $\eta_{i}\in G_{\U}$ for every index $i\leq n$.\\

{\bfseries Claim:} $P(\eta_{1},...,\eta_{n},\beta_{1},...,\beta_{m})=0$.\\

In fact, 
\begin{center} $P(\eta_{1},...,\eta_{n},\beta_{1},...,\beta_{m})=\sum\limits_{i=1}^{n}a_{i}\eta_{i}Q_{F_{i}}(\beta_{1},...,\beta_{m})=$\\$\sum\limits_{i=1}^{n}a_{i}\alpha_{i}(\prod\limits_{j\notin F_{i}}\beta_{j})(\prod\limits_{j\in F_{i}}\beta_{j})=\sum\limits_{i=1}^{n}a_{i}\alpha_{i}(\prod\limits_{j=1}^{m}\beta_{j})=(\prod\limits_{j=1}^{m}\beta_{j})\sum\limits_{i=1}^{n}a_{i}\alpha_{i}=0$.\end{center}
This shows that, if we set $x_{i}=\eta_{i}$ for $i=1,...,n$ and $y_{j}=\beta_{j}$ for $j=1,...,m$, we have an injective solution to $P(x_{1},...,x_{n},y_{1},...,y_{m})$ in $G_{\U}$.\end{proof}
Three observations: 
\begin{enumerate}
	\item as a consequence of the argument used to prove the theorem, the ultrafilter $\U$ considered in the proof is both a $\iota_{P}$-ultrafilter and a $\iota_{R}$-ultrafilter;
	\item we observe that some of the variables $y_{1},...,y_{m}$ may appear in more than a monomial: e.g., the polynomial \begin{center}$P(x_{1},x_{2},x_{3},x_{4},x_{5},y_{1},y_{2},y_{3})=x_{1}y_{1}y_{2}+4x_{2}y_{1}y_{2}y_{3}-3x_{3}y_{3}-2x_{4}y_{1}+x_{5}$\end{center} satisfies the hypothesis of Theorem \ref{lev}, so it is injectively partition regular;
	\item Theorem \ref{ConsHind} is a particular case of Theorem \ref{lev}.
\end{enumerate}
Theorem \ref{lev} can be reformulated in a way that leads to the generalization given by Theorem \ref{NLP}:
\begin{defn} Let $$P(x_{1},...,x_{n})=\sum\limits_{i=1}^{k} a_{i}M_{i}(x_{1},...,x_{n})$$ be a polynomial and let
$M_{1}(x_{1},...,x_{n}),...,M_{k}(x_{1},...,x_{n})$ be the distinct monic monomials of $P(x_{1},...,x_{n})$. We say that $\{v_{1},...,v_{k}\}\subseteq V(P)$ is a {\bfseries set of exclusive variables} for $P(x_{1},...,x_{n})$ if, for every $i,j\leq k$, $d_{M_{i}}(v_{j})\geq 1\Leftrightarrow i=j$.\\
In this case we say that the variable $v_{i}$ is {\bfseries exclusive} for the monomial $M_{i}(x_{1},...,x_{n})$ in $P(x_{1},...,x_{n})$.\end{defn}
E.g., the polynomial $P(x,y,z,t,w): xyz+yt-w$ admits $\{x,t,w\}$ or $\{z,t,w\}$ as sets of exlusive variables, while the polynomial $P(x,y,z): xy+yz-xz$ does not have any exclusive variable.
\begin{defn} Let $$P(x_{1},...,x_{n})=\sum\limits_{i=1}^{k} a_{i}M_{i}(x_{1},...,x_{n})$$ be a polynomial, and let
$M_{1}(x_{1},...,x_{n}),...,M_{k}(x_{1},...,x_{n})$ be the distinct monic monomials of $P(x_{1},...,x_{n})$. We call {\bfseries reduct of $P$} $($notation Red$(P))$ the polynomial:
	\begin{center} Red$(P)(y_{1},...,y_{k})=\sum\limits_{i=1}^{k} a_{i}y_{i}$. \end{center}\end{defn}
E.g., if $P(x,y,z,t,w)$ is the polynomial $x_{1}x_{2}+4x_{2}x_{3}-2x_{4}+x_{2}x_{5}$, then 
\begin{center} Red($P$)$(y_{1},y_{2},y_{3},y_{4})=y_{1}+4y_{2}-2y_{3}+y_{4}$.\end{center}
As a consequence of Rado's Theorem, we have that $P(x_{1},...,x_{n})$ satisfies Rado's condition if and only if $Red(P)$ is partition regular.
As a consequence of Theorem \ref{lev}, we obtain the following result:
\begin{cor} Let $n\geq 3$, $k\geq n$ be natural numbers and let $$P(x_{1},...,x_{n})=\sum\limits_{i=1}^{k}a_{i}M_{i}(x_{1},...,x_{n})$$ be a l.e.v. polynomial. We suppose that $P(x_{1},...,x_{n})$ admits a set of exclusive variables and that it satisfies Rado's Condition. Then $P(x_{1},...,x_{n})$ is an injectively partition regular polynomial.
\end{cor}
\begin{proof} If $n=k$, the polynomial is linear and the thesis follows by Theorem \ref{Rado}. So we can suppose that $k>n$. By reordering, if necessary, we can suppose that, for $j=1,...,k$, the variable $x_{j}$ is exclusive for the monomial $M_{j}(x_{1},...,x_{n})$. Then, by Rado's condition, the polynomial
\begin{center} $\sum\limits_{i=1}^{k}a_{i}x_{i}$ \end{center}
is partition regular. If $F=\{1,...,n-k\}$, for $i\leq k$ we set
\begin{center} $F_{i}=\{j\in F\mid x_{j+k}$ divides $M_{i}(x_{1},...,x_{n})\}.$\end{center}
Then if we set, for $j\leq n-k$, $y_{j}=x_{i+k}$, the polynomial $P(x_{1},...,x_{n})$ is, by renaming the variables, equal to
\begin{center} $\sum\limits_{i=1}^{k}a_{i}x_{i}Q_{F_{i}}(y_{1},...,y_{n-k})$. \end{center}
By Theorem \ref{lev} the above polynomial is injectively partition regular, so we have the thesis.\end{proof}
Corollary 3.7 talks about l.e.v. polynomials; in section 4 we show that there are also non l.e.v. polynomials that are partition regular, provided that they have "enough exclusive variables" in each monomial.
\section{Partition Regularity for a Class of Nonlinear Polynomials}
In this section we want to extend Theorem \ref{lev} to a particular class of nonlinear polynomials. To introduce our main result, we need the following notations:
\begin{defn} Let $P(x_{1},...,x_{n})=\sum\limits_{i=1}^{k} a_{i}M_{i}(x_{1},...,x_{n})$ be a polynomial, and let $M_{1}(x_{1},...,x_{n}),...,M_{k}(x_{1},...,x_{n})$ be the monic monomials of $P(x_{1},...,x_{n})$. Then
	\begin{itemize}
	\item {\bfseries NL(P)}=$\{x\in V(P)\mid d(x)\geq 2\}$ is the set of nonlinear variables of $P(x_{1},...,x_{n})$;
	\item for every $i\leq k$, $\mathbf{l_{i}}=\max\{d(x)-d_{i}(x)\mid x\in NL(P)\}$. 
	\end{itemize}
\end{defn}
\begin{thm}\label{NLP} Let $$P(x_{1},...,x_{n})=\sum\limits_{i=1}^{k} a_{i}M_{i}(x_{1},...,x_{n})$$ be a polynomial, and let $M_{1}(x_{1},...,x_{n})$,...,$M_{k}(x_{1},...,x_{n})$ be the monic monomials of $P(x_{1},...,x_{n})$. We suppose that $k\geq 3$, that $P(x_{1},...,x_{n})$ satisfies Rado's Condition and that, for every index $i\leq k$, in the monomial $M_{i}(x_{1},...,x_{n})$ there are at least $m_{i}=\max\{1,l_{i}\}$ exclusive variables with degree equal to 1.\\
Then $P(x_{1},...,x_{n})$ is injectively partition regular. \end{thm}
\begin{proof} We rename the variables in $V(P)$ in the following way: for $i\leq k$ let $x_{i,1},...,x_{i,m_{i}}$ be $m_{i}$ exclusive variables for $M_{i}(x_{1},...,x_{n})$ with degree equal to 1. We set 
\begin {center} $E=\{x_{i,j}\mid i\leq k, j\leq m_{i}\}$ \end{center}
and $NL(P)=\{y_{1},...,y_{h}\}$. Finally, we set $\{z_{1},...,z_{r}\}=V(P)\setminus (E\cup NL(P))$.\\
We suppose that the variables are ordered as to have 
\begin{center} $P(x_{1},...,x_{n})=P(x_{1,1},...,x_{1,m_{1}},x_{2,1},....,x_{k,m_{k}},z_{1},...,z_{r},y_{1},...,y_{h})$.\end{center}
We set
\begin{center} $\widetilde{P}(x_{1,1},...,z_{r})=P(x_{1,1},...,z_{r},1,1,...,1)$. \end{center}
By construction, and by hypothesis, $\widetilde{P}(x_{1,1},...,z_{r})$ is a l.e.v. polynomial with at least three monomials, it satisfies Rado's Condition and it has at least one exclusive variable for each monomial. So, by Theorem \ref{lev}, it is injectively partition regular. Let $\U$ be a multiplicative idempotent ultrafilter such that in $G_{\U}$ there is an injective solution $(\alpha_{1,1},...,\alpha_{k,m_{k}},\beta_{1},...,\beta_{r})$ to $\widetilde{P}(x_{1,1},...,z_{r})$.\\		
Let $\gamma$ be an element in $G_{\U}\setminus\{\alpha_{1,1},...,\alpha_{k,m_{k}},\beta_{1},...,\beta_{r}\}$.\\
We consider 
\begin{center} $\eta=\prod\limits_{i=1}^{h} S_{i}(\gamma)^{d(y_{i})}$. \end{center}
For $i=1,...,k$ we set $M^{NL}_{i}=\prod\limits_{j=1}^{h} S_{j}(\gamma)^{d_{i}(y_{j})}$ and 
\begin{center} $\eta_{i}=\frac{\eta}{M^{NL}_{i}}=\prod\limits_{j=1}^{h}S_{j}(\gamma)^{d(y_{j})-d_{i}(y_{j})}$. \end{center}
We observe that the maximum degree of an element $S_{j}(\gamma)$ in $\eta_{i}$ is, by construction, $l_{i}$.\\
Finally, for $1\leq j\leq m_{i}$, we set $I_{i,j}=\{s\leq h\mid d(y_{s})-d_{i}(y_{s})\geq j\}$ and 
\begin{center} $\gamma_{i,j}=\prod\limits_{s\in I_{i,j}} S_{s}(\gamma)$. \end{center}
With these choices, we have 
\begin{center} $\prod\limits_{j=1}^{m_{i}} \gamma_{i,j}=\eta_{i}$ \end{center}
and, by construction, $\{\gamma_{i,j}\mid i\leq k, j\leq m_{i}\}\subseteq G_{\U}$ since $\U$ is a multiplicative idempotent.\\
We also observe that, for every $i\leq k$, $\left(\prod\limits_{j=1}^{m_{i}}\gamma_{i,j}\right)\cdot M^{NL}_{i}=\eta$.\\
Now, if we set, for $i\leq k$ and $j\leq m_{i}$:
\begin{equation*} x_{i,j}= \begin{cases} \alpha_{i,j}\cdot\gamma_{i,j} & \mbox{if}  \ l_{i}\geq 1;\\
 \\
                  \alpha_{i,j} & \mbox{if} \ l_{i}=0; \end{cases}\end{equation*}
and
\begin{itemize}                 
\item $y_{i}=S_{i}(\gamma)$ for $i\leq h$;
\item $z_{i}=\beta_{i}$ for $i\leq r$
\end{itemize}
then  
\begin{center} $P(x_{1,1},...,x_{k,m_{k}},z_{1},...,z_{r},y_{1},...,y_{h})=$\\\vspace{0.3cm}$\eta\cdot\widetilde{P}(\alpha_{1,1},...,\alpha_{k,m_{k}},\beta_{1},...,\beta_{r},1,...,1)=0,$\end{center} 
so $P(x_{1,1},...,y_{h})$ is injectively partition regular. \\\end{proof}
In order to understand the requirement $k\geq 3$, we observe that one of the crucial points in the proof is that, when we set $y=1$ for every $y\in NL(P)$, the polynomial $\widetilde{P}(x_{1,1},...,z_{r})$ that we obtain is injectively partition regular. Now, let us suppose that $k=2$, and let $M_{1}(x_{1},...,x_{n})$ and $M_{2}(x_{1},...,x_{n})$ be the two monic monomials of $P(x_{1},...,x_{n})$. If $D(x_{1},...,x_{n})$ is the greatest common divisor of $M_{1}(x_{1},...,x_{n})$, $M_{2}(x_{1},...,x_{n})$, we set
\begin{equation*} Q_{i}(x_{1},...,x_{n})=\frac{M_{i}(x_{1},...,x_{n})}{D(x_{1},...,x_{n})} \end{equation*}
for $i=1,2$. We have
\begin{equation*} P(x_{1},...,x_{n})=D(x_{1},...,x_{n})(Q_{1}(x_{1},...,x_{n})-Q_{2}(x_{1},...,x_{n})), \end{equation*}
and it holds that $P(x_{1},...,x_{n})$ is injectively partition regular if and only if $R(x_{1},...,x_{n})=Q_{1}(x_{1},...,x_{n})-Q_{2}(x_{1},...,x_{n})$ is, since $D(x_{1},...,x_{n})$ is a nonzero monomial.\\
Now there are two possibilities: 
\begin{enumerate}
	\item $NL(R)\neq\emptyset$, in which case, since every $y\in NL(R)$ divides $Q_{1}(x_{1},...,x_{n})$ if and only if it does not divide $Q_{2}(x_{1},...,x_{n})$ (this property holds because, by construction, $Q_{1}(x_{1},...,x_{n})$ and $Q_{2}(x_{1},...,x_{n})$ are relatively prime), in at least one of the monomials there are at least two exclusive variables, and this entails that the polynomial $\widetilde{R}(x_{1,1},...,z_{r})$ is injectively partition regular by Theorem \ref{lev};
	\item $NL(R)=\emptyset$, in which case $R(x_{1},...,x_{n})$ is a l.e.v. polynomial with only two monomials, so it is injectively partition regular if and only if $n\geq 3$.
\end{enumerate}
By the previous discussion (and using the same notations) it follows that, when $k=2$, if the other hypothesis of Theorem \ref{NLP} hold then the polynomial $P(x_{1},...,x_{n})$ is injectively partition regular if and only if there do not exist two variables $x_{i},x_{j}\in V(P)$ such that $R(x_{1},...,x_{n})=x_{i}-x_{j}$.\\
We conclude this section by showing with an example how the proof of Theorem \ref{NLP} works. Consider the polynomial
 \begin{center} $P(x_{1,1},x_{2,1},x_{2,2},x_{3,1},x_{3,2},x_{4,1},x_{4,2},z_{1},z_{2},y_{1},y_{2})=$\\\vspace{0.2cm} $x_{1,1}y_{1}^{2}y_{2}^{2}+x_{2,1}x_{2,2}z_{1}y_{2}^{2}-2x_{3,1}x_{3,2}z_{2}y_{1}+x_{4,1}x_{4,2}$,\end{center}
where we have chosen the names of the variables following the notations introduced in the proof of Theorem \ref{NLP}.\\
We set 
\begin{center} $\widetilde{P}(x_{1,1},...,x_{4,2},z_{1},z_{2})=x_{1,1}+x_{2,1}x_{2,2}z_{1}-2x_{3,1}x_{3,2}z_{2}+x_{4,1}x_{4,2}$.\end{center} Let $\U$ be a multiplicative idempotent $\iota_{\widetilde{P}}$-ultrafilter and let $\alpha_{1,1},...,\alpha_{4,2},\beta_{1},\beta_{2}\in$$^{*}\N$ be mutually distinct elements in $G_{\U}$ such that 
\begin{center} $\alpha_{1,1}+\alpha_{2,1}\alpha_{2,2}\beta_{1}-2\alpha_{3,1}\alpha_{3,2}\beta_{2}+\alpha_{4,1}\alpha_{4,2}$. \end{center}
We take $\gamma\in G_{\U}\setminus\{\alpha_{1,1},...,\alpha_{4,2},\beta_{1},\beta_{2}\}$ and we set:
\begin{center} $\gamma_{2,1}=\gamma_{2,2}=$$^{*}\gamma$, $\gamma_{3,1}=$$^{*}\gamma$$^{**}\gamma$, $\gamma_{3,2}=$$^{**}\gamma$, $\gamma_{4,1}=\gamma_{4,2}=$$^{*}\gamma$$^{**}\gamma$.\end{center}
Finally, we set:
\begin{itemize}
\item $x_{1,1}=\alpha_{1,1}$;
\item $x_{2,1}=\alpha_{2,1}\cdot \gamma_{2,1}$;
\item $x_{2,2}=\alpha_{2,2}\cdot\gamma_{2,2}$;
\item $x_{3,1}=\alpha_{3,1}\cdot\gamma_{3,1}$;
\item $x_{3,2}=\alpha_{3,2}\cdot\gamma_{3,2}$;
\item $x_{4,1}=\alpha_{4,1}\cdot\gamma_{4,1}$;
\item $x_{4,2}=\alpha_{4,2}\cdot\gamma_{4,2}$;
\item $z_{1}=\beta_{1}$;
\item $z_{2}=\beta_{2}$;
\item $y_{1}=$$^{*}\gamma$;
\item $y_{2}=$$^{**}\gamma$.
\end{itemize}
With these choices, we have:
\begin{center} $P(x_{1,1},x_{2,1},x_{2,2},x_{3,1},x_{3,2},x_{4,1},x_{4,2},z_{1},z_{2},y_{1},y_{2})=$\\\vspace{0.2cm}$\alpha_{1,1}\cdot$$^{*}\gamma^{2}\cdot$$^{**}\gamma^{2}+\alpha_{2,1}\alpha_{2,2}\beta_{1}$$^{*}\gamma^{2}$$^{**}\gamma^{2}-2\alpha_{3,1}\alpha_{3,2}\beta_{2}$$^{*}\gamma^{2}$$^{**}\gamma^{2}+\alpha_{4,1}\alpha_{4,2}$$^{*}\gamma^{2}$$^{**}\gamma^{2}=$\\\vspace{0.2cm}$^{*}\gamma^{2}$$^{**}\gamma^{2}\widetilde{P}(\alpha_{1,1},\alpha_{2,1},\alpha_{2,2},\alpha_{3,1},\alpha_{3,2},\alpha_{4,1},\alpha_{4,2},\beta_{1},\beta_{2})=0$, \end{center}
so $P(x_{1,1},x_{2,1},x_{2,2},x_{3,1},x_{3,2},x_{4,1},x_{4,2},z_{1},z_{2},y_{1},y_{2})$ has an injective solution in $G_{\U}$, and this entails that $P(x_{1,1},x_{2,1},x_{2,2},x_{3,1},x_{3,2},x_{4,1},x_{4,2},z_{1},z_{2},y_{1},y_{2})$ is injectively partition regular.
\section{Conclusions}
A natural question is the following: can the implications in Theorem \ref{lev} and/or Theorem \ref{NLP} be reversed? The hypothesis on the existence of exclusive variables is not necessary: in \cite{rif6} it is proved that the polynomial
\begin{center} $P(x,y,z)=xy+xz-yz$ \end{center}
is partition regular (it can be proved that it is injectively partition regular), and it does not admit a set of exclusive variables. The hypothesis regarding Rado's Condition is more subtle: by slightly modifying the original arguments of Richard Rado (that can be found, for example, in \cite{rif10}) we can prove that this hypothesis is necessary for every homogeneous partition regular polynomial, but it seems to be not necessary in general. For sure, it is not necessary if we ask for the partition regularity of polynomials on $\Z$: in fact, e.g., the polynomial
\begin{center} $P(x_{1},x_{2},x_{3},y_{1},y_{2})=x_{1}y_{1}+x_{2}y_{2}+x_{3}$ \end{center}
is injectively partition regular on $\Z$ even if it does not satisfy Rado's Condition. This can be easily proved in the following way: the polynomial
\begin{center} $R(x_{1},x_{2},x_{3},y_{1},y_{2})=x_{1}y_{1}+x_{2}y_{2}-x_{3}$ \end{center} 
is, by Theorem \ref{lev}, injectively partition regular on $\N$ and, if $\mathfrak{U}_{\alpha}$ is a $\iota_{R}$-ultrafilter, then $\mathfrak{U}_{-\alpha}$ is a $\iota_{P}$-ultrafilter; in fact, if $\alpha_{1},\alpha_{2},\alpha_{3},\beta_{1},\beta_{2}$ are elements in $G_{\mathfrak{U}_{\alpha}}$ such that $R(\alpha_{1},\alpha_{2},\alpha_{3},\beta_{1},\beta_{2})=0$, then $-\alpha_{1},-\alpha_{2},-\alpha_{3},-\beta_{1},-\beta_{2}$ are elements in $G_{\mathfrak{U}_{-\alpha}}$ and, by construction, 
\begin{center} $P(-\alpha_{1},-\alpha_{2},-\alpha_{3},-\beta_{1},-\beta_{2})=0$.\end{center}
Furthermore, the previous example also shows that, while in the homogeneous case every polynomial which is partition regular on $\Z$ is also partition regular on $\N$, in the non homogeneous case this is false, since $P(x_{1},x_{2},x_{3},y_{1},y_{2})$, having only positive coefficients, can not be partition regular on $\N$ (it does not even have any solution in $\N)$.\\
Finally, Rado's Condition alone is not sufficient to ensure the partition regularity of a nonlinear polynomial: in \cite{rif6} the authors proved that the polynomial 
\begin{center} $x+y-z^{2}$ \end{center}
is not partition regular on $\N$, even if it satisfies Rado's Condition.\\
We conclude the paper summarizing the previous observations in two questions:\\

{\bfseries Question 1:} Is there a characterization of nonlinear partition regular polynomials on $\N$ in "Rado's Style", i.e. that allows to determine in a finite time if a given polynomial $P(x_{1},...,x_{n})$ is, or is not, partition regular?\\

Question 1 seems particularly challenging; an easier question, that would still be interesting to answer, is the following:\\

{\bfseries Question 2:} Is there a characterization of homogeneous partition regular polynomials (in the same sense of Question 1)? 
\section{Acknowledgements} 
I would like to thank Professor Mauro di Nasso for many reasons: first of all, I became interested in problems regarding combinatorial number theory under his supervision; the ideas behind the techniques exposed in this work were originated by his idea of characterizing properties of ultrafilters thinking about them as nonstandard points (e.g., the characterization of idempotent ultrafilters given in Proposition \ref{idultragen}); finally, he gave me many useful comments regarding the earlier draft of the paper.

\end{document}